\documentclass{article}
\usepackage{amsmath,amssymb}
\usepackage{enumerate}

\newtheorem{theorem}{Theorem}[section]
\newtheorem{lemma}[theorem]{Lemma}
\newtheorem{proposition}[theorem]{Proposition}

\newtheorem{definition}[theorem]{Definition}
\newtheorem{remark}[theorem]{Remark}

\newcommand{\N}{\mathbb{N}}

\newcommand{\id}{\operatorname{id}}
\newcommand{\sym}{\operatorname{Sym}}
\newcommand{\aut}{\operatorname{Aut}}

\newcommand{\ord}[1]{\vert #1\vert}

\newcommand{\dotcup}{\mathop{\mathaccent\cdot\cup}}

\newenvironment{proof}{\par\noindent{\bf Proof.}}{$\qed$\par\bigskip}
\newcommand{\qed}{\enspace\vrule  height6pt  width4pt  depth2pt}
\begin{document}

\title{A family of solutions of the Yang-Baxter equation}

\author{David Bachiller \and Ferran Ced\'o}
\date{}
\maketitle

\begin{abstract}
A new method to construct involutive non-degenerate set-theoretic
solutions $(X^n,r^{(n)})$ of the Yang-Baxter equation from an
initial solution $(X,r)$ is given. Furthermore, the permutation
group $\mathcal{G}(X^n,r^{(n)})$ associated to the solution
$(X^n,r^{(n)})$ is isomorphic to a subgroup of $\mathcal{G}(X,r)$,
and in many cases $\mathcal{G}(X^n,r^{(n)})\cong \mathcal{G}(X,r)$.
\end{abstract}



\section{Introduction}

The quantum Yang-Baxter equation is one of the basic equations in
ma\-the\-ma\-ti\-cal physics named after the authors of the two
first works in which the equation arose: the solution of the delta
function Fermi gas by C. N. Yang \cite{Yang}, and the solution of
the 8-vertex model by R. J. Baxter \cite{Baxter}. It also lies at
the foundation of the theory of quantum groups. One of the important
open problems related to this equation is compute all its solutions.
Those are linear maps $R: V\otimes V\to V\otimes V$, with $V$ a
vector space, that satisfy
$$R_{12}\circ R_{13}\circ R_{23}=R_{23}\circ R_{13}\circ R_{12},$$
where $R_{ij}$ denotes the  map $R_{ij}: V\otimes V\otimes V\to
V\otimes V\otimes V$ acting as $R$ on the $(i,j)$ tensor factors and
as the identity on the remaining factor.

Finding all the solutions of  the Yang-Baxter equation is a
difficult task far from being solved. Nevertheless, many
solutions have been found during the last 20 years and the related
algebraic structures (Hopf algebras) have been studied.

In \cite{Drinfeld}, Drinfeld suggested the study of  a simpler case:
solutions induced by a linear extension of a mapping $\mathcal{R}:
X\times X\to X\times X$, where $X$ is a basis for $V$. In this case,
one says that $\mathcal{R}$ is a set-theoretic solution of the
quantum Yang-Baxter equation. It is not difficult to see that, if
$\mathcal{\tau}:X^2\to X^2$ is the map defined by
$\mathcal{\tau}(x,y)=(y,x)$, then the map $\mathcal{R}:X^2\to X^2$
is a set-theoretic solution of the quantum Yang-Baxter equation if
and only if the mapping $r=\mathcal{\tau}\circ\mathcal{R}$ is a
solution of the equation
$$
r_{12}\circ r_{23}\circ r_{12}=r_{23}\circ r_{12}\circ r_{23},
$$
where $r_{ij}$ is the map from $X^3$ to $X^3$ that acts as $r$
on the $(i,j)$ components and as the identity on the
remaining component. In  the sequel, we will always work with this
last equivalent equation.

We study solutions with some additional conditions:
involutively and non-degeneracy. A map
$$
\begin{array}{cccc}
r:& X\times X&  \longrightarrow&    X\times X\\
        &   (x,y)&  \longmapsto&    (\sigma_x(y),\gamma_y(x))
\end{array}
$$
is said to be involutive if $r\circ r=\id_{X^2}$. Moreover, it  is
said to be left (resp. right) non-degenerate if each map
$\sigma_x$ (respectively, $\gamma_y$) is bijective, and it is said
to be non-degenerate if it is left and right non-degenerate. If $r$
is involutive and left non-degenerate, it can be checked that
it satisfies $r_{12}\circ r_{23}\circ r_{12}=r_{23}\circ r_{12}\circ
r_{23}$ if and only if it satisfies
$\sigma_x\circ\sigma_{\sigma^{-1}_x(y)}=\sigma_y\circ\sigma_{\sigma^{-1}_y(x)}$
for all $x,y\in X$ (see the proof of
\cite[Theorem~9.3.10]{JO}).

In what follows, by a solution of the YBE we will mean a
non-degenerate involutive set-theoretic solution of the Yang-Baxter
equation.

In the last years, solutions of the YBE have  received a lot of
attention
\cite{Cedo1,Cedo2,Etingof,Gateva,GatevaCameron,GatevaMajid,GatevaVandenBergh,Jespers,Lu,Rump}.
 In this case each solution $(X,r)$ of the YBE has an associated
structure group, denoted by $G(X,r)$, and defined by
$$
G(X,r):=\left\langle x\in X\mid xy=zt\text{ if and only if } r(x,y)=(z,t)\right\rangle.
$$
When $X$ is finite, groups isomorphic to some $G(X,r)$ are called
groups of $I$-type. There is another important group associated to
every solution of the YBE, its permutation group $\mathcal{G}(X,r)$,
which is the subgroup of $\sym_X$ generated by the bijections
$\sigma_x$, for all $x\in X$. It can be proved that
$\mathcal{G}(X,r)$ is a homomorphic image of $G(X,r)$. When $X$ is
finite, groups isomorphic to some $\mathcal{G}(X,r)$ are called IYB
groups.

In \cite{Cedo1}, in order to characterize the groups of
$I$-type, it is suggested to follow two steps:
\bigskip

\noindent {\bf Step 1:}  Determine the finite groups that are IYB
groups.

\noindent {\bf Step 2:} Given an IYB group $G$, find all the
solutions of the YBE $(Y,s)$ with $Y$ finite  such that
$\mathcal{G}(Y,s)\cong G$.
\bigskip

\noindent Nowadays, these two problems remain unsolved. In the
recent Ph.D. thesis of Nir Ben David \cite{NirBenDavid}, it is
claimed that the result corresponding to \cite[Corollary
D]{NirBenDavid} solves Step 2 in homological terms.  In fact, this
result reduces Step 2 to the following problem.
\bigskip

\noindent {\bf Problem.} {\em Let $G$ be an IYB group. Let
$\pi\colon G\longrightarrow A$ be a bijective $1$-cocycle over a
$G$-module $A$. Let $n$ be a positive integer. Find all the
extensions of $G$-modules
$$0\longrightarrow \mathbb{Z}^n\longrightarrow E\longrightarrow A\longrightarrow 0,$$
where $\mathbb{Z}^n$ is a trivial $G$-module, $E\cong \mathbb{Z}^n$
as abelian groups and there is a basis $Y$ of $E$, as free abelian
group, which is invariant by the action of $G$ on $E$.}
\bigskip

The aim of this work is  to present some new results related to the
resolution of Step 2. To this end, an algebraic structure called
brace and introduced by Rump in \cite{Rump} is very useful. Rump
showed that this structure has deep connections with solutions of
the YBE. We use some of this connections to prove our results.

This paper is organized as follows. In Section \ref{preliminars}, we
recall some results about braces. We only sketch there, without
proofs, the theorems that we need later; for further reading, an
introduction to braces and their connection to the Yang-Baxter
equation can be found in \cite{Cedo2}. Next, we devote Section
\ref{seccioconstruccio} to present our main result about solutions
of the YBE with fixed permutation group. We do not solve the general
problem, but we generalize a result of \cite{Cedo1} which gives a
non-obvious construction of an infinite family of solutions with a
fixed permutation group. Specifically, given an initial solution
$(X,r)$ of the YBE, we state and prove a procedure to define, for
each $n\in\N$, a solution $(X^n,r^{(n)})$  such that
$\mathcal{G}(X^n, r^{(n)})$ is isomorphic to a subgroup of
$\mathcal{G}(X,r)$, and then we provide sufficient conditions for
 $\mathcal{G}(X^n,r^{(n)})$ to be isomorphic to $\mathcal{G}(X,r)$. One of
the key steps of the proof of this result  is the use of the
properties of the brace structure.

\section{Preliminars about braces}\label{preliminars}

We only present here very briefly the results about braces that we
will need later.  For a much more detailed account, see
\cite{Cedo2}. We begin recalling the definition of left brace.

\begin{definition}
{\rm A \emph{left brace}  is a set $G$ with two operations $+$ and
$\cdot$ such that $(G,+)$ is an abelian group, $(G,\cdot)$ is a
group, and every $a,b,c\in G$ satisfy
$$
a(b+c)+a=ab+ac.
$$
We will refer to this property as the brace property.  We call
$(G,+)$ the additive group, and $(G,\cdot)$ the multiplicative group
of the left brace. Right braces are defined similarly, changing the
brace property by $(b+c)a+a=ba+ca$.}
\end{definition}

For any $a\in G$, we define a map $\lambda_a: G\to G$ by $\lambda_a(b)=ab-a$.
In the study of braces, these maps play an important role; here is a list of some of their properties.
\begin{lemma}\label{proplambda}
Let $G$ be a left brace. The following properties hold:
\begin{itemize}
\item[(i)] $\lambda_a$ is bijective and $\lambda^{-1}_a=\lambda_{a^{-1}}$.
\item[(ii)] $\lambda_a(x+y)=\lambda_a(x)+\lambda_a(y)$; that is, $\lambda_a$ is an automorphism of the abelian group $(G,+)$.
\item[(iii)] $\lambda_a\lambda_b=\lambda_{ab}$; that is, the map $\lambda:(G,\cdot)\to \aut(G,+)$, defined by $\lambda(a)=\lambda_a$ is a homomorphism of groups.
\item[(iv)] $a+b=a\cdot\lambda^{-1}_a(b).$
\item[(v)] $a\cdot\lambda^{-1}_a(b)=b\cdot\lambda^{-1}_b(a).$
\item[(vi)] $\lambda_a\lambda_{\lambda^{-1}_a(b)}=\lambda_b\lambda_{\lambda^{-1}_b(a)}$.
\item[(vii)] The map $r:G\times G\to G\times G$ defined by  $r(x,y)=\left(\lambda_x(y),\lambda^{-1}_{\lambda_x(y)}(x)\right)$
is a solution of the YBE. It is called \emph{solution associated to
the left brace $G$}.
\end{itemize}
\end{lemma}
\begin{proof}
See \cite[Lemmas 2.9 and 4.1]{Cedo2}.
\end{proof}

So any left brace gives us a solution of the YBE.
There are other relations between braces and solutions of the YBE,
like the two next results, which are a characterization of groups of
I-type and IYB groups through braces.

\begin{proposition}
A group is isomorphic to $\mathcal{G}(X,r)$ for some  solution of
the YBE $(X,r)$ if and only if it is the multiplicative group of a
left brace.

In particular, a finite group $G$ is an IYB group if and only if it is the multiplicative group of a finite left brace.
\end{proposition}
\begin{proof}
See \cite[Corollary 4.6]{Cedo2}.
\end{proof}

\begin{proposition}\label{prop5.2}
A group $G$ is of $I$-type if and only if it is isomorphic to the
multiplicative group of a left brace $B$ such that the additive
group of $B$ is a free abelian group with a finite basis $X$ such
that $\lambda_x(y)\in X$, for all $x,y\in X$.
\end{proposition}
\begin{proof}
See \cite[Proposition 5.2]{Cedo2}.
\end{proof}

The next theorem is an essential tool for the proof of the main
result of this paper. It allows us to embed any solution of the YBE
$(X,s)$ inside a left brace, and then we can use all the
additional algebraic properties of this structure.

\begin{theorem}\label{thm4.4}
Let $(X,s)$ be a solution of the YBE. Then $G(X,s)$ is isomorphic to
the multiplicative group of a left brace $H$ such that, if $(H,r)$
is the solution associated to it, then there exists a subset $Y$ of
$H$ such that $(Y,r')$, where $r'$ is the restriction of $r$ to
$Y^2$, is a solution of the YBE isomorphic to $(X,r)$. Furthermore,
$\lambda_h(y)\in Y$, for all $y\in Y$ and all $h\in H$.
\end{theorem}
\begin{proof}
See \cite[Theorem 4.4]{Cedo2} and its proof.
\end{proof}

\section{Solutions with fixed permutation group}\label{seccioconstruccio}
In this section, we focus on the problem of the construction of
solutions of the YBE with a fixed permutation group.

There is a simple way to produce solutions with the same
permutation group. Note that, for any set $X$, the map
$r(x,y)=(y,x)$, for all $x,y\in X$, is a solution of the YBE
with trivial permutation group; we call $(X,r)$ the
trivial solution on $X$. Note also that, if we have solutions
$(X_i,r_i)$ of the YBE with permutation group
$G_i=\mathcal{G}(X_i,r_i)$, for all $i=1,\dots,n$, and
$X=\displaystyle{\dotcup}_{i=1}^n X_i$ is the disjoint union of the
sets $X_i$, then $(X,r)$ is a solution of the YBE, where
$$
r(x,y)=\left\lbrace
\begin{array}{ll}
r_i(x,y)&   \text{if there exists an } i \text{ such that } x,y\in X_i,\\
(y,x)&          \text{otherwise}
\end{array}
\right.
$$
with permutation group isomorphic to $G_1\times\cdots\times
G_n$. Hence, since $\{\id\}\times G\cong G$, we can associate
 to each solution $(X,r)$ of the YBE infinitely many
solutions of the YBE with the same permutation group. However, this
construction only increases the size of our set adding points that
behave as the trivial solution, so we want to find ``less trivial''
constructions.

We will generalize \cite[Lemma~5.2]{Cedo1}, which gives a
non-obvious construction of an infinite family of solutions of the
YBE associated to a fixed permutation group. This result is
stated in \cite{Cedo1} in terms of cycle sets. Translated to the
language of solutions of the YBE equation, the result could be
stated as follows.

\begin{proposition}\label{5.2Cedo1}
Let $(X,r)$ be a solution of the YBE, with $r(x,y)=(
\sigma_x(y),\gamma_{y}(x))$. Then, for $x_1,x_2\in X$, the map
$f_{(x_1,x_2)}\colon X^2\longrightarrow X^2$ defined by
$$
f_{(x_1,x_2)}(y_1,y_2)=(
\sigma_{x_1}\sigma_{x_2}(y_1),\sigma^{-1}_{\sigma_{x_1}\sigma_{x_2}(y_1)}\sigma_{x_1}\sigma_{x_2}\sigma_{y_1}(y_2)),
$$
for $y_1,y_2\in X$, is bijective and $(X^2,r^{(2)})$, where
$r^{(2)}:X^2\times X^2\longrightarrow X^2\times X^2$ is the map
defined by
$$
r^{(2)}((x_1,x_2),(y_1,y_2))=( f_{(x_1,x_2)}(y_1,y_2),
f^{-1}_{f_{(x_1,x_2)}(y_1,y_2)}(x_1,x_2)),
$$
is a solution of the YBE. Moreover, if $\sigma_z=\id$ for some $z\in
X$, then $\mathcal{G}(X^2,r^{(2)})\cong \mathcal{G}(X,r)$.
\end{proposition}

\begin{remark}
{\rm The assumption that $\sigma_x=\id$, for some $x\in X$, is not
true in general.  However, increasing the size of our set, we can
always construct a solution with this property preserving the same
 permutation group. Namely, consider the solution $(X\dotcup
\{z\},r')$ for the disjoint union of $X$ with a set of one point,
defined as $r'(x,y)=r(x,y)$ if $x$ and $y$ belong to $X$ and
$r'(x,y)=(y,x)$ if either $x=z$ or $y=z$ (note that it is the
same trivial construction given at the beginning of this section).}
\end{remark}

Thus, beginning with an initial solution $(X,r)$ of the YBE and
applying Proposition~\ref{5.2Cedo1} to $X$, $X^2$, $(X^2)^2$\dots~we
can construct infinitely many solutions of the YBE with the same
associated permutation group $\mathcal{G}(X,r)$.

 In order to generalize this result we will need the following
lemma, which is a generalization of \cite[Lemma~5.1]{Cedo1}. We use
the notation $\overline{x}=(x_1,\dots,x_n)$ for the elements of
$X^n$.

\begin{lemma}\label{psi}
Let $X$ be a set and let $\sigma\colon X\longrightarrow \sym_X$ be a
map. Consider the map $\psi:  \sym_X \longrightarrow    \sym_{X^n}$
defined by
$\psi(\tau)(\overline{y})=\left(t_1^\tau(\overline{y}),\dots,t_n^{\tau}(\overline{y})\right)$,
for all $\tau\in\sym_X$ and all $\overline{y}\in X^n$, where
$
t_1^{\tau}(\overline{y})=\tau(y_1),
$
$$
t_{j+1}^\tau(\overline{y})=\sigma(t_j^\tau(\overline{y}))^{-1}\cdots\sigma(t_1^\tau(\overline{y}))^{-1}
\tau\sigma(y_1)\cdots\sigma(y_{j})(y_{j+1})\text{, for }j>0.
$$
Then $\psi$ is a monomorphism.
\end{lemma}

\begin{proof} Let $\tau\in \sym_X$.
It is easy to check that $\psi(\tau)$ is a bijective map from $X^n$
to $X^n$ with inverse
$$
(y_1,~\dots~,y_n)\mapsto
(T_1^\tau(\overline{y}),~\dots~,T_n^\tau(\overline{y})),
$$
where
$
T_1^\tau(\overline{y})=\tau^{-1}(y_1),
$
$$
T_{j+1}^\tau(\overline{y})=\sigma(T_j^\tau(\overline{y}))^{-1}\cdots\sigma(T_1^\tau(\overline{y}))^{-1}
\tau^{-1}\sigma(y_1)\cdots\sigma(y_{j})(y_{j+1}).
$$
Thus $\psi$ is well-defined.

To prove that $\psi$ is a morphism, we have to check
$\psi(\tau\circ\xi)=\psi(\tau)\circ\psi(\xi)$ for all
$\tau,\xi\in\sym_X$. Component by component, this is equivalent to
verify

$$
t_j^{\tau\circ\xi}(y_1,\dots,y_n)=t_j^{\tau}(t_1^{\xi}(\overline{y}),\dots,t_n^\xi(\overline{y})),
~~\forall j=1,\dots,n.
$$
This is done by induction. The first component is almost immediate;
we write
$\overline{t^\xi}(\overline{y})=(t_1^{\xi}(\overline{y}),\dots,t_n^\xi(\overline{y}))$
for short:
$$
t_1^{\tau\circ\xi}(\overline{y})=\tau(\xi(y_1))=\tau(t_1^\xi(\overline{y}))=t_1^\tau(\overline{t^{\xi}}(\overline{y})).
$$
Now, assume that we have checked it up to the $j$-th component, and
we want to prove it for the $(j+1)$-th component:
\begin{eqnarray*} t_{j+1}^{\tau\circ\xi}(\overline{y})&=&
\sigma(t_j^{\tau\xi}(\overline{y}))^{-1}\cdots\sigma(t_1^{\tau\xi}(\overline{y}))^{-1}
\tau\xi\sigma(y_1)\cdots\sigma(y_{j})(y_{j+1})\\
&=&\sigma(t_j^{\tau}(\overline{t^\xi}(\overline{y})))^{-1}\cdots\sigma(t_1^{\tau}(\overline{t^\xi}(\overline{y})))^{-1}\tau\xi\sigma(y_1)\cdots\sigma(y_{j})(y_{j+1})\\
&=&\sigma(t_j^{\tau}(\overline{t^\xi}(\overline{y})))^{-1}\cdots\sigma(t_1^{\tau}(\overline{t^\xi}(\overline{y})))^{-1}\tau\\
&&\cdot\sigma(t^{\xi}_1(\overline{y}))\cdots\sigma(t_j^\xi(\overline{y}))
\sigma(t^{\xi}_j(\overline{y}))^{-1}\cdots\sigma(t_1^\xi(\overline{y}))^{-1}
\xi\sigma(y_1)\cdots\sigma(y_{j})(y_{j+1})\\
&=&
\sigma(t_j^{\tau}(\overline{t^\xi}(\overline{y})))^{-1}\cdots\sigma(t_1^{\tau}(\overline{t^\xi}(\overline{y})))^{-1}
\tau\\
&&\cdot\sigma(t^{\xi}_1(\overline{y}))\cdots\sigma(t_j^\xi(\overline{y}))
t_{j+1}^\xi(\overline{y}) =t_{j+1}^\tau
(\overline{t^\xi}(\overline{y})). \end{eqnarray*} The second
equality comes from the induction hypothesis, and at the end we use
the definition of $t_{j+1}^\xi(\overline{y})$.

On the other hand, to prove that $\psi$ is injective,  suppose that
$\psi(\tau)=\psi(\xi)$ for some $\tau,\xi\in\sym_X$. Then,
$\psi(\tau)(\overline{y})=\psi(\xi)(\overline{y})$ for all
$\overline{y}\in X^n$. Looking at the first component, we get
$\tau(y_1)=\xi(y_1)$ for all $y_1\in X$, so $\tau=\xi$.
\end{proof}

The next two results give the announced generalization of
\cite[Lemma~5.2]{Cedo1}.

\begin{theorem}\label{teoremaXn}
Let $(X,r)$ be a solution of the YBE, with $r(x,y)=(
\sigma_x(y),\gamma_{y}(x))$. Let $n$ be an integer greater that
$1$. For $\overline{x}\in  X^n$, consider the map
$f_{\overline{x}}\colon X^n\longrightarrow X^n$ defined by
$$
f_{\overline{x}}(\overline{y})=(
h_1(\overline{x},\overline{y}),h_2(\overline{x},\overline{y}),\dots,h_n(\overline{x},\overline{y})),
$$
for $\overline{y}\in X^n$, where the $h_j$ is defined recursively by
$$
h_1(\overline{x},\overline{y})=\sigma_{x_1}\cdots\sigma_{x_n}(y_1),
$$
and
$$
h_j(\overline{x},\overline{y})=\sigma^{-1}_{h_{j-1}(\overline{x},\overline{y})}\cdots\sigma^{-1}_{h_{1}(\overline{x},\overline{y})}
\sigma_{x_1}\cdots\sigma_{x_n}\sigma_{y_1}\cdots\sigma_{y_{j-1}}(y_j),
$$
for $j=2,\dots, n$. Then $f_{\overline{x}}$ is bijective and
$(X^n,r^{(n)})$, where $r^{(n)}:X^n\times X^n\longrightarrow
X^n\times X^n$ is the map defined by
$$
r^{(n)}(\overline{x},\overline{y})=( f_{\overline{x}}(\overline{y}),
f^{-1}_{f_{\overline{x}}(\overline{y})}(\overline{x})),
$$
is a solution of the YBE.
\end{theorem}

\begin{proof}
Let $\sigma\colon X\longrightarrow \sym_X$ be the map defined
by $\sigma(x)=\sigma_x$, for all $x\in X$. Consider the map
$\psi\colon \sym_X\longrightarrow \sym_{X^n}$ defined as in
Lemma~\ref{psi}. It is clear that
$f_{\overline{x}}=\psi(\sigma_{x_1}\cdots\sigma_{x_n})$. Hence
$f_{\overline{x}}$ is bijective.

By Theorem~\ref{thm4.4}, we may assume that $X$ is a subset of a
left brace $H$ and that $\sigma_x$ is the restriction of $\lambda_x$
to $X$, for all $x\in X$. Recall that $\lambda_a(b)=ab-a$, for all
$a,b\in H$. Therefore, by Lemma~\ref{proplambda}(iii),
$$
h_1(\overline{x},\overline{y})=\sigma_{x_1}\cdots\sigma_{x_n}(y_1)=\lambda_{x_1\cdots
x_n}(y_1).
$$
We claim that
\begin{equation}\label{3.1}
\lambda_{x_1\cdots x_n} (y_1\cdots
y_j)=h_1(\overline{x},\overline{y})\cdots
h_{j}(\overline{x},\overline{y}),
\end{equation}
for all $1\leq j\leq n$ and $x_1,\dots x_n,y_1,\dots ,y_n\in X$.

We prove this claim by induction on $j$. We know that it is true for
$j=1$. Suppose that $j>1$ and the claim is true for $j-1$. We have
\begin{eqnarray*}
\lambda_{x_1\cdots x_n} (y_1\cdots y_j)&=&\lambda_{x_1\cdots x_n}
(y_1\cdots y_{j-1}+\lambda_{y_1\cdots y_{j-1}}(y_j))\\
&=&\lambda_{x_1\cdots x_n}
(y_1\cdots y_{j-1})+\lambda_{x_1\cdots x_n}\lambda_{y_1\cdots y_{j-1}}(y_j)\\
&&\quad \mbox{(by Lemma~\ref{proplambda}(ii))}\\
&=&h_1(\overline{x},\overline{y})\cdots
h_{j-1}(\overline{x},\overline{y})+\lambda_{x_1\cdots x_n}\lambda_{y_1\cdots y_{j-1}}(y_j)\\
&&\quad \mbox{(by induction hypothesis)}\\
&=&h_1(\overline{x},\overline{y})\cdots
h_{j-1}(\overline{x},\overline{y})\lambda^{-1}_{h_1(\overline{x},\overline{y})\cdots
h_{j-1}(\overline{x},\overline{y})}\lambda_{x_1\cdots x_n}\lambda_{y_1\cdots y_{j-1}}(y_j)\\
&&\quad \mbox{(by Lemma~\ref{proplambda}(iv))}\\
\end{eqnarray*}
By Lemma~\ref{proplambda}(iii),
$$\lambda^{-1}_{h_1(\overline{x},\overline{y})\cdots
h_{j-1}(\overline{x},\overline{y})}\lambda_{x_1\cdots
x_n}\lambda_{y_1\cdots
y_{j-1}}(y_j)=h_j(\overline{x},\overline{y}).$$ Hence the claim
follows. By (\ref{3.1}), we have that
\begin{eqnarray}\label{hj}
&&h_{j}(\overline{x},\overline{y})=\lambda_{x_1\cdots x_n}(y_1\cdots
y_{j-1})^{-1}\lambda_{x_1\cdots x_n}(y_1\cdots y_{j}),
\end{eqnarray}
for $2\leq j\leq n$. Let
$H_1(\overline{x},\overline{y})=\lambda^{-1}_{x_1\cdots x_n}(y_1)$
and
$$H_j(\overline{x},\overline{y})=\lambda^{-1}_{x_1\cdots x_n}(y_1\cdots
y_{j-1})^{-1}\lambda^{-1}_{x_1\cdots x_n}(y_1\cdots y_{j}),$$ for
$2\leq j\leq n$. It is easy to check that the map
$X^n\longrightarrow X^n$ defined by $\overline{y}\mapsto
(H_1(\overline{x},\overline{y}),\dots
,H_n(\overline{x},\overline{y}))$ is $f^{-1}_{\overline{x}}$.

By the definition of $r^{(n)}$, it is straightforward to check that
$r^{(n)}\circ r^{(n)}=\id$. Thus, in order to prove that
$(X^n,r^{(n)})$  is an involutive non-degenerate set-theoretic
solution of the Yang-Baxter equation, we should show that
\begin{itemize}
\item[(a)]
$
f_{\overline{x}}f_{f^{-1}_{\overline{x}}(\overline{y})}(\overline{z})=
f_{\overline{y}}f_{f^{-1}_{\overline{y}}(\overline{x})}(\overline{z}),\text{
for all }\overline{x},\overline{y},\overline{z}\in X^n,$ and
\item[(b)] the map $\gamma_{\overline{y}}\colon X^n\longrightarrow X^n$ defined by
$\gamma_{\overline{y}}(\overline{x})=
f^{-1}_{f_{\overline{x}}(\overline{y})}(\overline{x})$ is bijective.
\end{itemize}
By (\ref{3.1}) and the definition of
$H_j(\overline{x},\overline{y})$, the first component of
$f_{\overline{x}}f_{f^{-1}_{\overline{x}}(\overline{y})}(\overline{z})=f_{\overline{x}}f_{(H_1(\overline{x},\overline{y}),\dots,H_n(\overline{x},\overline{y}))}(\overline{z})$
is \begin{eqnarray*} \lambda_{x_1\cdots
x_n}\lambda_{H_1(\overline{x},\overline{y})\cdots
H_n(\overline{x},\overline{y})}(z_1)&=& \lambda_{x_1\cdots
x_n}\lambda_{\lambda^{-1}_{x_1\cdots x_n}(y_1\cdots y_n)}(z_1)\\
&=& \lambda_{y_1\cdots y_n}\lambda_{\lambda^{-1}_{y_1\cdots
y_n}(x_1\cdots x_n)}(z_1),
\end{eqnarray*}
where the last equality follows from
Lemma~\ref{proplambda}(vi). For $j>1$, the $j$-th component of
$f_{\overline{x}}f_{f^{-1}_{\overline{x}}(\overline{y})}(\overline{z})=f_{\overline{x}}f_{(H_1(\overline{x},\overline{y}),\dots,H_n(\overline{x},\overline{y}))}(\overline{z})$
is
$$
( \lambda_{x_1\cdots
x_n}\lambda_{H_1(\overline{x},\overline{y})\cdots
H_n(\overline{x},\overline{y})}(z_1\cdots z_{j-1}))^{-1} \cdot (
\lambda_{x_1\cdots x_n}\lambda_{H_1(\overline{x},\overline{y})\cdots
H_n(\overline{x},\overline{y})}(z_1\cdots z_{j}))
$$
$$
=(\lambda_{x_1\cdots x_n}\lambda_{\lambda^{-1}_{x_1\cdots
x_n}(y_1\cdots y_n)}(z_1\cdots z_{j-1}))^{-1} \cdot
(\lambda_{x_1\cdots x_n}\lambda_{\lambda^{-1}_{x_1\cdots
x_n}(y_1\cdots y_n)}(z_1\cdots z_{j}))
$$
and, by Lemma~\ref{proplambda}(vi), we can interchange the $x$'s and
the $y$'s, and (a) follows.

To prove (b), first we shall see that $\gamma_{\overline{y}}$ is
injective. Let $\overline{x},\overline{z}\in X^n$ be elements such
that
$\gamma_{\overline{y}}(\overline{x})=\gamma_{\overline{y}}(\overline{z})$.
Hence $H_j((h_1(\overline{x},\overline{y}),\dots
,h_n(\overline{x},\overline{y})),\overline{x})=H_j((h_1(\overline{z},\overline{y}),\dots
,h_n(\overline{z},\overline{y})),\overline{z})$, for all $j=1,\dots,
n$. That is $$\lambda^{-1}_{h_1(\overline{x},\overline{y})\cdots
h_n(\overline{x},\overline{y})}(x_1)=\lambda^{-1}_{h_1(\overline{z},\overline{y})\cdots
h_n(\overline{z},\overline{y})}(z_1)$$ and
$$\lambda^{-1}_{h_1(\overline{x},\overline{y})\cdots
h_n(\overline{x},\overline{y})}(x_1\cdots
x_{j-1})^{-1}\lambda^{-1}_{h_1(\overline{x},\overline{y})\cdots
h_n(\overline{x},\overline{y})}(x_1\cdots x_j)$$
$$=\lambda^{-1}_{h_1(\overline{z},\overline{y})\cdots
h_n(\overline{z},\overline{y})}(z_1\cdots
z_{j-1})^{-1}\lambda^{-1}_{h_1(\overline{z},\overline{y})\cdots
h_n(\overline{z},\overline{y})}(z_1\cdots z_j),$$ for all $j=2,\dots
,n$. Therefore
$$\lambda^{-1}_{h_1(\overline{x},\overline{y})\cdots
h_n(\overline{x},\overline{y})}(x_1\cdots
x_j)=\lambda^{-1}_{h_1(\overline{z},\overline{y})\cdots
h_n(\overline{z},\overline{y})}(z_1\cdots z_j),$$ for all $j=1,\dots
,n$. By (\ref{3.1}),
\begin{eqnarray}\label{inj}&&\lambda^{-1}_{\lambda_{x_1\cdots x_n}(y_1\cdots
y_n)}(x_1\cdots x_j)=\lambda^{-1}_{\lambda_{z_1\cdots z_n}(y_1\cdots
y_n)}(z_1\cdots z_j),\end{eqnarray}
for all $j=1,\dots ,n$. By
Lemma~\ref{proplambda}(vii), since
$$\lambda^{-1}_{\lambda_{x_1\cdots x_n}(y_1\cdots y_n)}(x_1\cdots
x_n)=\lambda^{-1}_{\lambda_{z_1\cdots z_n}(y_1\cdots y_n)}(z_1\cdots
z_n),$$ we have that $x_1\cdots x_n=z_1\cdots z_n$. Thus, by
(\ref{inj}), $x_1\cdots x_j=z_1\cdots z_j$, for all $j=1,\dots ,n$.
Hence $x_j=z_j$, for all $j=1,\dots ,n$. Therefore
$\gamma_{\overline{y}}$ is injective.

We shall see that $\gamma_{\overline{y}}$ is surjective. Let
$\overline{z}=(z_1,\dots ,z_n)\in X^n$. By
Lemma~\ref{proplambda}(vii), there exists $a_n\in H$ such that
$\lambda^{-1}_{\lambda_{a_n}(y_1\cdots y_n)}(a_n)=z_1\cdots z_n$.
Let $$a_j=\lambda_{\lambda_{a_n}(y_1\cdots y_n)}(z_1\dots z_j),$$
for all $j=1,\dots ,n-1$. By Theorem~\ref{thm4.4}, $a_1\in X$. Let
$x_1=a_1$ and $x_i=a_{i-1}^{-1}a_i$, for $1<i\leq n$. We shall prove
that $x_i\in X$, for all $i$, and
$\gamma_{\overline{y}}(\overline{x})=\overline{z}$. Suppose that
$i>1$ and $x_1,\dots ,x_{i-1}\in X$. We have that
\begin{eqnarray*}a_i&=&\lambda_{\lambda_{a_n}(y_1\cdots y_n)}(z_1\dots z_i)=\lambda_{\lambda_{a_n}(y_1\cdots y_n)}(z_1\dots z_{i-1}+\lambda_{z_1\dots
z_{i-1}}(z_i))\\
&=&\lambda_{\lambda_{a_n}(y_1\cdots y_n)}(z_1\dots
z_{i-1})+\lambda_{\lambda_{a_n}(y_1\cdots y_n)}(\lambda_{z_1\dots
z_{i-1}}(z_i))\quad (\mbox{by Lemma~\ref{proplambda}(ii)}) \\
&=&a_{i-1}+\lambda_{\lambda_{a_n}(y_1\cdots y_n)z_1\dots
z_{i-1}}(z_i)\quad (\mbox{by Lemma~\ref{proplambda}})(iii) \\
&=&a_{i-1}\lambda^{-1}_{a_{i-1}}(\lambda_{\lambda_{a_n}(y_1\cdots
y_n)z_1\dots z_{i-1}}(z_i))\quad (\mbox{by
Lemma~\ref{proplambda}})(iv).
\end{eqnarray*}
Hence
$x_i=a_{i-1}^{-1}a_i=\lambda^{-1}_{a_{i-1}}(\lambda_{\lambda_{a_n}(y_1\cdots
y_n)z_1\dots z_{i-1}}(z_i))\in X$, by Theorem~\ref{thm4.4}. Thus, by
induction, $x_1,\dots ,x_n\in X$. Now we have
$$\gamma_{\overline{y}}(\overline{x})=(H_1((h_1(\overline{x},\overline{y}),\dots, h_n(\overline{x},\overline{y})),\overline{x}),\dots,
H_n((h_1(\overline{x},\overline{y}),\dots,
h_n(\overline{x},\overline{y})),\overline{x})).$$ By (\ref{3.1}) and
the definition of $H_j$,
\begin{eqnarray*}
\gamma_{\overline{y}}(\overline{x})&=&(\lambda^{-1}_{\lambda_{x_1\cdots
x_n}(y_1\cdots y_n)}(x_1), \lambda^{-1}_{\lambda_{x_1\cdots
x_n}(y_1\cdots y_n)}(x_1)^{-1} \lambda^{-1}_{\lambda_{x_1\cdots
x_n}(y_1\cdots y_n)}(x_1x_2),\dots,\\
&& \lambda^{-1}_{\lambda_{x_1\cdots x_n}(y_1\cdots y_n)}(x_1\cdots
x_{n-1})^{-1} \lambda^{-1}_{\lambda_{x_1\cdots x_n}(y_1\cdots
y_n)}(x_1\cdots x_n)
)\\
&=& (\lambda^{-1}_{\lambda_{a_n}(y_1\cdots y_n)}(a_1),
\lambda^{-1}_{\lambda_{a_n}(y_1\cdots y_n)}(a_1)^{-1}
\lambda^{-1}_{\lambda_{a_n}(y_1\cdots y_n)}(a_2),\dots,\\
&&\lambda^{-1}_{\lambda_{a_n}(y_1\cdots y_n)}(a_{n-1})^{-1}
\lambda^{-1}_{\lambda_{a_n}(y_1\cdots y_n)}(a_n) )=(z_1,z_2,\dots
,z_n). \end{eqnarray*} Therefore $\gamma_{\overline{y}}$ is
bijective and (b) follows. This finishes the proof.
\end{proof}

The next proposition describes the permutation group of the
solutions $(X^n,r^{(n)})$ as a certain subgroup of the
permutation group of $(X,r)$. After that, we give two cases in
which $\mathcal{G}(X^n,r^{(n)})\cong \mathcal{G}(X,r)$.

\begin{proposition}
With the above notation, $\mathcal{G}(X^n,r^{(n)})$ is
isomorphic to the subgroup of $\mathcal{G}(X,r)$ generated by all
the permutations $\sigma_{x_1}\cdots\sigma_{x_n}$, for all $x_i\in
X$. In particular,
\begin{enumerate}
\item If $\sigma_z=\id$ for some $z\in X$,
then $\mathcal{G}(X^n,r^{(n)})\cong \mathcal{G}(X,r)$ for all $n$.
\item If $X$ is a finite set and $\ord{\mathcal{G}(X,r)}=m<+\infty$, then,
 for all $n$ such that
$\gcd(m,n)=1$, we have $\mathcal{G}(X^n,r^{(n)})\cong \mathcal{G}(X,r)$.
\end{enumerate}
\end{proposition}
\begin{proof}
Recall that $r(x,y)=(\sigma_x(y),\gamma_y(x))$. Let $\sigma\colon
X\longrightarrow \sym_X$ be the map defined by $\sigma(x)=\sigma_x$,
for all $x\in X$. Consider the map $\psi\colon \sym_X\longrightarrow
\sym_{X^n}$ defined as in Lemma~\ref{psi}. We have that

\begin{eqnarray*} \mathcal{G}(X^n,r^{(n)})&=&\left\langle f_{\overline{x}} :
\overline{x}\in X^n\right\rangle= \left\langle
\psi(\sigma_{x_1}\cdots\sigma_{x_n}) : x_i\in X\right\rangle\\
&=&\psi(\left\langle \sigma_{x_1}\cdots\sigma_{x_n} : x_i\in
X\right\rangle)\cong \left\langle \sigma_{x_1}\cdots\sigma_{x_n} :
x_i\in X\right\rangle, \end{eqnarray*} using in the last isomorphism
that $\psi$ is  a monomorphism.

In particular,
\begin{enumerate}
\item If $\sigma_z=\id$ for some $z\in X$, then any $\sigma_x$ can be written as a product of
$n$ permutations using $\sigma_x=\sigma_z^{n-1}\sigma_x$, so
$\left\langle \sigma_{x_1}\cdots\sigma_{x_n} : x_i\in
X\right\rangle= \left\langle \sigma_x : x\in X\right\rangle$, and

$$
\mathcal{G}(X^n,r^{(n)})\cong
\left\langle \sigma_{x_1}\cdots\sigma_{x_n} : x_i\in X\right\rangle=
\left\langle \sigma_x : x\in X\right\rangle=\mathcal{G}(X,r).
$$

\item Suppose that $X$ is finite, and that $n$ is a positive integer coprime with $m=\ord{\mathcal{G}(X,r)}$.
We can find a positive integer $k$ such that $nk\equiv 1\pmod{m}$.
Then, $\sigma_x=\sigma_x^{nk}$ for all $x\in X$, which implies $\left\langle
\sigma_{x_1}\cdots\sigma_{x_n} : x_i\in X\right\rangle= \left\langle
\sigma_x : x\in X\right\rangle$, and the conclusion follows as in
the previous case.
\end{enumerate}

\end{proof}

\section*{Acknowledgments}
 Research partially supported by DGI MINECO MTM2011-28992-C02-01, by FEDER
UNAB10-4E-378 "Una manera de hacer Europa", and by the Comissionat
per Universitats i Recerca de la Generalitat de Catalunya.

\vspace{30pt}
 \noindent \begin{tabular}{llllllll}
 D. Bachiller && F. Ced\'{o}  \\
 Departament de Matem\`atiques &&  Departament de Matem\`atiques \\
 Universitat Aut\`onoma de Barcelona &&  Universitat Aut\`onoma de Barcelona  \\
08193 Bellaterra (Barcelona), Spain    && 08193 Bellaterra
(Barcelona), Spain \\
dbachiller@mat.uab.cat && cedo@mat.uab.cat
\end{tabular}

\end{document}